\documentclass[12pt]{amsart}
\usepackage{a4wide}
\usepackage{graphicx,eepic, bm, times, bbm}
\usepackage{amsthm,amsmath,amsfonts,amssymb}
\usepackage{multirow}
\newtheorem{lemma}{Lemma}[section]

\newtheorem{prop}[lemma]{Proposition}
\newtheorem{thm}[lemma]{Theorem}
\newtheorem{cor}[lemma]{Corollary}
\theoremstyle{definition}

\newtheorem{conjecture}[lemma]{Conjecture}

\theoremstyle{remark}
\newtheorem{remark}[lemma]{Remark}

\numberwithin{equation}{section} \numberwithin{table}{section}

\begin{document}

\title[Approximation properties of $\beta$-expansions II]{Approximation properties of $\beta$-expansions II}
\author{Simon Baker}
\address{
School of Mathematics, The University of Manchester,
Oxford Road, Manchester M13 9PL, United Kingdom. E-mail:
simonbaker412@gmail.com}

\date{\today}

\begin{abstract}
Given $\beta\in(1,2)$ and $x\in[0,\frac{1}{\beta-1}]$, a sequence $(\epsilon_{i})_{i=1}^{\infty}\in\{0,1\}^{\mathbb{N}}$ is called a $\beta$-expansion for $x$ if $$x=\sum_{i=1}^{\infty}\frac{\epsilon_{i}}{\beta^{i}}.$$ In a recent article the author studied the quality of approximation provided by the finite sums $\sum_{i=1}^{n}\epsilon_{i}\beta^{-i}$ \cite{Bak}. In particular, given $\beta\in(1,2)$ and $\Psi:\mathbb{N}\to\mathbb{R}_{\geq 0},$ we associate the set $$W_{\beta}(\Psi):=\bigcap_{m=1}^{\infty}\bigcup_{n=m}^{\infty}\bigcup_{(\epsilon_{i})_{i=1}^{n}\in\{0,1\}^{n}}\Big[\sum_{i=1}^{n}\frac{\epsilon_{i}}{\beta^{i}},\sum_{i=1}^{n}\frac{\epsilon_{i}}{\beta^{i}}+\Psi(n)\Big].$$ Alternatively, $W_{\beta}(\Psi)$ is the set of $x\in \mathbb{R}$ such that for infinitely many $n\in\mathbb{N},$ there exists a sequence $(\epsilon_{i})_{i=1}^{n}$ satisfying the inequalities
$$0\leq x-\sum_{i=1}^{n}\frac{\epsilon_{i}}{\beta^{i}}\leq \Psi(n).$$ If $\sum_{n=1}^{\infty}2^{n}\Psi(n)<\infty$ then $W_{\beta}(\Psi)$ has zero Lebesgue measure. We call a $\beta\in(1,2)$ approximation regular, if $\sum_{n=1}^{\infty}2^{n}\Psi(n)=\infty$ implies $W_{\beta}(\Psi)$ is of full Lebesgue measure within $[0,\frac{1}{\beta-1}]$. The author conjectured in \cite{Bak} that Lebesgue almost every $\beta\in(1,2)$ is approximation regular. In this paper we make a significant step towards proving this conjecture.

The main result of this paper is the following statement: given a sequence of positive real numbers $(\omega_{n})_{n=1}^{\infty},$ which satisfy $\lim_{n\to\infty} \omega_{n}=\infty$, then for Lebesgue almost every $\beta\in(1.497\ldots,2)$ the set $W_{\beta}(\omega_{n}\cdot 2^{-n})$ is of full Lebesgue measure within $[0,\frac{1}{\beta-1}]$. Here the sequence $(\omega_{n})_{n=1}^{\infty}$ should be interpreted as a sequence tending to infinity at a very slow rate.

We also study the case where $\sum_{n=1}^{\infty}2^{n}\Psi(n)<\infty$. Applying the mass transference principle developed by Beresnevich and Velani in \cite{BerVel}, we prove some results on the Hausdorff dimension and Hausdorff measure of $W_{\beta}(\Psi)$.

\end{abstract}

\maketitle

\section{Introduction}
Expansions in non-integer bases were pioneered in the late 1950's with the papers of Parry \cite{Parry} and Renyi \cite{Renyi}. Since then they have been studied by many authors and have connections with ergodic theory, fractal geometry, and number theory(see the survey articles \cite{Kom} and \cite{Sid2}). In this article we study expansions in non-integer bases from the perspective of Diophantine approximation and metric number theory.

Classical Diophantine approximation is concerned with the approximation properties of the rational numbers. These approximation properties are described via a limsup set in the following general way. Given $\Psi:\mathbb{N}\to\mathbb{R}_{\geq 0}$ we associate the set
\begin{align*}
J(\Psi):&=\bigcap_{n=1}^{\infty}\bigcup_{q=n}^{\infty}\bigcup_{p\in\mathbb{Z}}\Big[\frac{p}{q}-\Psi(q),\frac{p}{q}+\Psi(q)\Big]\\
&=\Big\{x\in\mathbb{R}: \Big|x-\frac{p}{q}\Big|\leq \Psi(q)\, \textrm{for infinitely many } (p,q)\in\mathbb{Z}\times\mathbb{N}\Big\}.
\end{align*}A well known theorem due to Khintchine states that if $\Psi$ is a non-increasing function and $\sum_{q=1}^{\infty}q\psi(q)=\infty$, then almost every $x\in\mathbb{R}$ is contained in $J(\Psi)$ \cite{Khit}. By the Borel-Cantelli lemma, if $\sum_{q=1}^{\infty}q\psi(q)<\infty$ then $J(\Psi)$ has zero Lebesgue measure. In a significant paper Duffin and Schaeffer showed that in the statement of Khitchine's theorem, it is not possible to remove all monotonicity assumptions on $\Psi$ \cite{DufSch}. They constructed a function $\Psi:\mathbb{N}\to\mathbb{R}_{\geq 0}$ for which $\sum_{q=1}^{\infty} q\Psi(q)=\infty,$ yet $J(\Psi)$ has zero Lebesgue measure. In the statement of Khintchine's theorem, almost every is meant with respect to the Lebesgue measure. Throughout this paper whenever we use the phrase ``almost every," we always mean with respect to the Lebesgue measure. We denote the Lebesgue measure by $\lambda(\cdot)$.

This article is motivated by the following question: $$\textrm{Does an analogue of Khintchine's theorem hold within expansions in non-integer bases?}$$ Before stating the appropriate analogue of Khintchine's theorem, it is necessary to recall some of the theory behind expansions in non-integer bases.

\subsection{Expansions in non-integer bases}
Let $\beta\in(1,2)$ and $I_{\beta}:=[0,\frac{1}{\beta-1}].$ Given $x\in I_{\beta}$ we call a sequence $(\epsilon_{i})_{i=1}^{\infty}\in\{0,1\}^{\mathbb{N}}$ a \emph{$\beta$-expansion} for $x$ if $$x=\sum_{i=1}^{\infty}\frac{\epsilon_{i}}{\beta^{i}}.$$ Despite being a simple generalisation of the well known integer base expansions, $\beta$-expansions exhibit very different behaviour. As an example, a well known theorem of Sidorov \cite{Sid} states that for any $\beta\in(1,2),$ almost every $x\in I_{\beta}$ has a continuum of $\beta$-expansions. This is of course completely different to the usual integer base expansions, where every number has a unique expansion except for a countable set of exceptions which have precisely two.

Given $\beta\in(1,2)$ and a sequence $(\epsilon_{i})_{i=1}^{n}\in\{0,1\}^{n},$ we call the number $\sum_{i=1}^{n}\epsilon_{i}\beta^{-i}$ the \emph{level $n$ sum corresponding to $(\epsilon_{i})_{i=1}^{n}\in\{0,1\}^{n}.$} Moreover, given $x\in I_{\beta}$ we call a sequence $(\epsilon_{i})_{i=1}^{n}\in\{0,1\}^{n}$ an \emph{$n$ prefix} for $x,$ if there exists $(\epsilon_{n+i})_{i=1}^{\infty}$ such that
\begin{equation}
\label{prefix equation}
x=\sum_{i=1}^{n}\frac{\epsilon_{i}}{\beta^{i}}+\sum_{i=1}^{\infty}\frac{\epsilon_{n+i}}{\beta^{n+i}}.
\end{equation} In other words, the sequence $(\epsilon_{i})_{i=1}^{n}$ is an $n$ prefix for $x$ if it can be extended to form a $\beta$-expansion of $x$. When a sequence $(\epsilon_{i})_{i=1}^{n}$ satisfies (\ref{prefix equation}), if there is no confusion we may refer to both the sequence $(\epsilon_{i})_{i=1}^{n}$ and the number $\sum_{i=1}^{n}\epsilon_{i}\beta^{-i}$ as an $n$ prefix for $x$. The building blocks of a $\beta$-expansion are the prefixes. As we will see below, the prefixes play the same role for us as the rational numbers do in traditional Diophantine approximation.

Given $\beta\in(1,2)$ and $\Psi:\mathbb{N}\to\mathbb{R}_{\geq 0}$ we associate the following limsup set: $$W_{\beta}(\Psi) :=\bigcap_{m=1}^{\infty}\bigcup_{n=m}^{\infty}\bigcup_{(\epsilon_{i})_{i=1}^{n}\in\{0,1\}^{n}}\Big[\sum_{i=1}^{n}\frac{\epsilon_{i}}{\beta^{i}},\sum_{i=1}^{n}\frac{\epsilon_{i}}{\beta^{i}}+\Psi(n)\Big].$$
Alternatively, $W_{\beta}(\Psi)$ is the set of $x\in \mathbb{R}$ such that for infinitely many $n\in\mathbb{N},$ there exists a level $n$ sum satisfying the inequalities
\begin{equation}
\label{approx equation}
0\leq x-\sum_{i=1}^{n}\frac{\epsilon_{i}}{\beta^{i}}\leq \Psi(n).
\end{equation} Our goal is to understand how well a generic $x\in I_{\beta}$ can be approximated by its prefixes. In (\ref{approx equation}) the approximation to $x$ is given by a level $n$ sum, not necessarily an $n$ prefix for $x$. However, as is explained in \cite{Bak} there is no loss of generality in assuming that the inequalities in (\ref{approx equation}) are satisfied by an $n$ prefix.

Our goal of understanding how well a generic $x\in I_{\beta}$ can be approximated by its prefixes, is in a sense equivalent to understanding how uniformly the level $n$ sums are distributed throughout $I_{\beta}$. If we can show that an optimal rate of approximation holds for a generic $x$, then the level $n$ sums should be distributed reasonably uniformly throughout $I_{\beta}$. Similarly, if the level $n$ sums are well distributed within $I_{\beta},$ we would expect to have good approximation properties. This behaviour was observed in \cite{Bak}. Understanding how the level $n$ sums are distributed within $I_{\beta}$ is a long standing and classical problem. For recent developments on this classical problem we refer the reader to an important paper by Hochman \cite{Hoc}, and the references therein. Given a $\beta\in(1,2),$ one method for understanding how these sums are distributed is to study the properties of the measure $\mu_{\beta},$ where $$\mu_{\beta}(E)= \mathbb{P}\Big(\Big\{(\epsilon_{i})_{i=1}^{\infty}\in\{0,1\}^{\mathbb{N}}: \sum_{i=1}^{\infty}\frac{\epsilon_{i}}{\beta^{i}}\in E\Big\}\Big),$$ for any Borel set $E\subseteq \mathbb{R}$. Here $\mathbb{P}$ is the $(1/2,1/2)$ probability measure on $\{0,1\}^{\mathbb{N}}.$ The measure $\mu_{\beta}$ is known as the \emph{Bernoulli convolution} with respect to $\beta$. If $\mu_{\beta}$ fails to be absolutely continuous with respect to the Lebesgue measure, then it is expected that the level $n$ sums will be distributed in a far from uniform way. This statement is supported by the fact that the only $\beta$ for which $\mu_{\beta}$ is known to not be absolutely continuous are the Pisot numbers, and in a Pisot base the level $n$ sums are poorly distributed. In a breakthrough paper, Solomyak showed that for almost every $\beta\in(1,2)$ the measure $\mu_{\beta}$ is absolutely continuous with respect to the Lebesgue measure \cite{Solomyak}. This was recently improved upon by Shmerkin who showed that the set of $\beta$ for which $\mu_{\beta}$ is not absolutely continuous has Hausdorff dimension zero \cite{Shm}. The study of the sets $W_{\beta}(\Psi)$ could be interpreted as an alternative method for understanding the distribution of the level $n$ sums.

As mentioned above, it is believed that if the level $n$ sums are well distributed then $\mu_{\beta}$ will be absolutely continuous. The author expects that a similar property would imply a Khintchine type result. The following definition introduces the set of $\beta\in(1,2)$ for which we have Khintchine type behaviour. We call a $\beta\in(1,2)$ \emph{approximation regular}, if $\sum_{n=1}^{\infty}2^{n}\Psi(n)=\infty$ implies $W_{\beta}(\Psi)$ is of full measure within $I_{\beta}$. Note that if $\sum_{n=1}^{\infty}2^{n}\Psi(n)<\infty,$ then $W_{\beta}(\Psi)$ will always have zero Lebesgue measure by the Borel-Cantelli lemma. For a generic $\beta$ we expect the level $n$ sums to be well distributed within $I_{\beta},$ as such we made the following conjecture in \cite{Bak}.

\begin{conjecture}
\label{conjecture 1}
Almost every $\beta\in(1,2)$ is approximation regular.
\end{conjecture}

Conjecture \ref{conjecture 1} is our analogue of Khintchine's theorem within expansions in non-integer bases. Note that any $\beta$ which satisfies a height one polynomial cannot be approximation regular. As such, there is a dense set of exceptions to Conjecture \ref{conjecture 1}. In \cite{Bak} we were able to verify Conjecture \ref{conjecture 1} for a class of algebraic integers known as \emph{Garsia numbers}. These numbers are the positive real algebraic integers with norm $\pm 2,$ whose conjugates are all of modulus strictly greater than $1.$ In particular the main result of \cite{Bak} is the following.
\begin{thm}
\label{Garsia thm}
All Garsia numbers are approximation regular.
\end{thm}
Note that Garsia showed that whenever $\beta$ is a Garsia number, the associate Bernoulli convolution $\mu_{\beta}$ is absolutely continuous with respect to the Lebesgue measure, and has bounded density \cite{Gar}. In this paper we make a significant step towards a proof of Conjecture \ref{conjecture 1}. Our main result is the following.

\begin{thm}
\label{Main theorem}
Let $(\omega_{n})_{n=1}^{\infty}$ be a sequence of real numbers that tend to infinity. Then for almost every $\beta\in(1.497\ldots,2)$ the set $W_{\beta}(\omega_{n}\cdot 2^{-n})$ is of full Lebesgue measure within $I_{\beta}$.
\end{thm} The quantity $1.497\ldots$ is a constant appearing as a consequence of transversality arguments used in \cite{Solomyak1}. The proof of Theorem \ref{Main theorem} relies heavily on counting estimates appearing within this paper. In the statement of Theorem \ref{Main theorem} the sequence $(\omega_{n})_{n=1}^{\infty}$  should be interpreted as a sequence that tends to infinity at a very slow rate. As an application of Theorem \ref{Main theorem} we obtain the following corollary.

\begin{cor}
\label{cor1}
For almost every $\beta\in(1.497\ldots,2),$ the set of $x\in I_{\beta}$ with infinitely many solutions to the inequalities $$0\leq x-\sum_{i=1}^{n}\frac{\epsilon_{i}}{\beta^{i}}\leq \frac{\log n}{2^{n}} ,$$ is of full Lebesgue measure within $I_{\beta}$.
\end{cor}
 \begin{remark} Note that any $\beta$ which satisfies a height one polynomial, as well as failing to be approximation regular, also satisfies $\lambda(W_{\beta}(\log n \cdot 2^{-n}))=0$. So there is a dense set of exceptions to Corollary \ref{cor1}. It is natural to ask about the set $I_{\beta}\setminus W_{\beta}(\log n \cdot 2^{-n}).$ We will see later that there exists $\beta$ for which $W_{\beta}(\log n\cdot 2^{-n})$ is of full Lebesgue measure within $I_{\beta}$, yet $I_{\beta}\setminus W_{\beta}(\log n\cdot 2^{-n})$ has positive Hausdorff dimension.
\end{remark}
 The remaining difficulty in proving Conjecture \ref{conjecture 1} is not knowing how to deal with $\Psi$ for which $\sum_{n=1}^{\infty}2^{n}\Psi(n)=\infty,$ yet $2^{n}\Psi(n)$ tends to zero. Theorem \ref{Main theorem} provides no answers in this case. However, Theorem \ref{Main theorem} does at least demonstrate that studying approximations of the order $2^{-n}$ is the correct action to take.

Given Theorem \ref{Main theorem}, it is natural to ask whether a single $\beta$-expansion $(\epsilon_{i})_{i=1}^{\infty}\in\{0,1\}^{\mathbb{N}}$ can satisfy $$0\leq x-\sum_{i=1}^{n}\frac{\epsilon_{i}}{\beta^{i}}\leq \frac{\omega_{n}}{2^{n}},$$ for infinitely many $n\in\mathbb{N}$? This happens whenever a mild technical condition is satisfied. We say that $(\omega_{n})_{n=1}^{\infty}$ is \textit{growing regularly}, if for each $m\in\mathbb{N}$ there exists $K_{m}\in\mathbb{N}$ such that
\begin{equation}
\label{decay reg 2}
\frac{\omega(n+m)}{\omega(n)}\geq \frac{1}{K_{m}}
\end{equation}holds for every $n\in\mathbb{N}.$ We emphasise that the constant $K_{m}$ is allowed to depend on $m$. Note that a sequence $(\omega_{n})_{n=1}^{\infty}$ satisfies the growing regularly property whenever $(\omega_{n})_{n=1}^{\infty}$ is increasing, simply take $K_{m}=1$ for all $m\in\mathbb{N}$. The condition $(\omega_{n})_{n=1}^{\infty}$ is \textit{growing regularly} is equivalent to the function $\Psi(n)=\omega_{n}\cdot 2^{-n}$ decaying regularly. Where decaying regularly is as in \cite{Bak}. We have the following theorem.
\begin{thm}
\label{decay theorem}
Let $(\omega_{n})_{n=1}^{\infty}$ be a sequence of real numbers that tend to infinity and are growing regularly. Then for almost every $\beta\in(1.497\ldots,2),$ the set of $x\in I_{\beta}$ that have a $\beta$-expansion $(\epsilon_{i})_{i=1}^{\infty}$ that satisfies $$0\leq x-\sum_{i=1}^{n}\frac{\epsilon_{i}}{\beta^{i}}\leq \frac{\omega_{n}}{2^{n}},$$ for infinitely many $n\in\mathbb{N},$ is of full Lebesgue measure within $I_{\beta}$.
\end{thm} Replicating the proof of Theorem $1.4$ from \cite{Bak}, we can use Theorem \ref{Main theorem} to prove Theorem \ref{decay theorem}. As this proof is a simple generalisation we do not include the details. As a consequence of Theorem \ref{decay theorem} we obtain the following corollary.
\begin{cor}
For almost every $\beta\in(1.497\ldots,2),$ the set of $x\in I_{\beta}$ that have a $\beta$-expansion $(\epsilon_{i})_{i=1}^{\infty}$ that satisfies $$0\leq x-\sum_{i=1}^{n}\frac{\epsilon_{i}}{\beta^{i}}\leq \frac{\log n}{2^{n}},$$ for infinitely many $n\in\mathbb{N},$ is of full Lebesgue measure within $I_{\beta}$.
\end{cor}

In Section $2$ we prove Theorem \ref{Main theorem}. In Section $3$ we study the Hausdorff dimension and Hausdorff measure of the set $W_{\beta}(\Psi)$ in the case when $\sum_{n=1}^{\infty}2^{n}\Psi(n)<\infty.$ In particular, we employ the mass transference principle of Beresnevich and Velani \cite{BerVel} to calculate these quantities for certain values of $\beta$.

Before moving onto our proof of Theorem \ref{Main theorem} we summarise the work of some other authors on the approximation properties of $\beta$-expansions. In \cite{PerRev} and \cite{PerRevA} Persson and Reeve consider a setup very similar to our own, they introduced the set  $$K_{\beta}(\Psi):=\bigcap_{m=1}^{\infty}\bigcup_{n=m}^{\infty}\bigcup_{(\epsilon_{i})_{i=1}^{n}\in\{0,1\}^{n}}\Big[\sum_{i=1}^{n}\frac{\epsilon_{i}}{\beta^{i}}-\Psi(n),\sum_{i=1}^{n}\frac{\epsilon_{i}}{\beta^{i}}+\Psi(n)\Big].$$ Notice that $W_{\beta}(\Psi)\subseteq K_{\beta}(\Psi).$ Our setup is slightly different because we are interested in the approximation properties of prefixes. Persson and Reeve restrict to the case where $\Psi(n)=2^{-n\alpha}$ for some $\alpha\in(1,\infty),$ so in particular $K_{\beta}(2^{-n\alpha})$ is of zero Lebesgue measure for any $\beta\in(1,2)$. Motivated by Falconer \cite{Falconer} they studied the intersection properties of $K_{\beta}(\Psi)$. In \cite{Falconer} Falconer defined $G^{s}$ to be the set of $A\subseteq \mathbb{R},$ which have the property that for any countable collection of similarities $\{f_{j}\}_{j=1}^{\infty},$ we have $$\dim_{H}\Big(\bigcap_{j=1}^{\infty}f_{j}(A)\Big)\geq s.$$ Persson and Reeve generalised the definition of $G^{s}$ to arbitrary intervals $I$ by defining $G^{s}(I):=\{A\subseteq I: A+diam(I)\mathbb{Z}\in G^{s}\}.$ The main results of \cite{PerRev} and  \cite{PerRevA} are summarised in the following theorem.
\begin{thm}
\label{PerRev thm}
Let $\alpha\in(1,\infty)$ and $\Psi(n)=2^{-\alpha n}$.
\begin{itemize}
  \item For all $\beta\in(1,2),$ $\dim_{H}(K_{\beta}(\Psi))\leq \frac{1}{\alpha}$.
  \item For almost every $\beta\in(1,2),$ $K_{\beta}(\Psi)\in G^{s}(I_{\beta})$ for $s=\frac{1}{\alpha}.$
  \item For a dense set of $\beta\in(1,2),$ $\dim_{H}(K_{\beta}(\Psi))<\frac{1}{\alpha}.$
  \item For all $\beta\in(1,2)$, $K_{\beta}(\Psi)\in G^{s}(I_{\beta})$ for $s=\frac{\log \beta}{\alpha\log 2}.$
  \item For a countable set of $\beta\in(1,2),$ $\dim_{H}( K_{\beta}(\Psi))=\frac{\log \beta}{\alpha\log 2}.$
\end{itemize}
\end{thm}
The approximation properties of $\beta$-expansions were also studied by Dajani, Komornik, Loreti, and de Vries in \cite{Daj}. Given $x\in I_{\beta},$ they call a sequence $(\epsilon_{i})_{i=1}^{\infty}$ an \emph{optimal expansion for $x$} if $(\epsilon_{i})_{i=1}^{\infty}$ is a $\beta$-expansion for $x,$ and if for every other $\beta$-expansion of $x$ the following inequality holds for every $n\in\mathbb{N}:$ $$x-\sum_{i=1}^{n}\frac{\epsilon_{i}}{\beta^{i}}\leq x-\sum_{i=1}^{n} \frac{\epsilon_{i}'}{\beta^{i}}.$$ In other words, for each $n\in\mathbb{N}$ the $n$ prefix $(\epsilon_{i})_{i=1}^{n}$ always provides the closest approximation. In \cite{Daj} the authors showed that every $x$ in $I_{\beta}$ has an optimal expansion if and only if $\beta$ is contained in a special class of algebraic integers known as the multinacci numbers. A \textit{multinacci number} is the unique root of an equation of the form $x^{n+1}=x^{n}+\cdots +x +1$ contained in $(1,2)$. The main result of \cite{Daj} is the following.

\begin{thm}
\begin{itemize}
  \item Let $\beta$ be a multinacci number, then every $x\in I_{\beta}$ has an optimal expansion.
  \item If $\beta\in (1,2)$ is not a multinacci number, then the set of $x\in I_{\beta}$ with an optimal expansion is nowhere dense and has zero Lebesgue measure.
\end{itemize}
\end{thm}
Note that the countable set of $\beta$ appearing in Theorem \ref{PerRev thm} for which $\dim_{H}( K_{\beta}(\Psi))=\frac{\log \beta}{\alpha\log 2}$ is precisely the set of multinacci numbers.

\section{Proof of Theorem \ref{Main theorem}}
\subsection{Normalisation to the unit interval} Before giving our proof of Theorem \ref{Main theorem} we normalise $I_{\beta}$ and the level $n$ sums, so we can focus our attention on the unit interval $[0,1]$. This will make some of our later calculations more straightforward. Given $\beta\in(1,2)$ and $\Psi:\mathbb{N}\to\mathbb{R}_{\geq 0},$ we introduce the set
$$V_{\beta}(\Psi):=\bigcap_{m=1}^{\infty}\bigcup_{n=m}^{\infty}\bigcup_{(\epsilon_{i})_{i=1}^{n}\in\{0,1\}^{n}}\Big[(\beta-1)\sum_{i=1}^{n}\frac{\epsilon_{i}}{\beta^{i}},(\beta-1)\sum_{i=1}^{n}\frac{\epsilon_{i}}{\beta^{i}}+\Psi(n)\Big].$$ By replacing the $\sum_{i=1}^{n}\epsilon_{i}\beta^{-i}$ term appearing in $W_{\beta}(\Psi)$ with a $(\beta-1)\sum_{i=1}^{n}\epsilon_{i}\beta^{-i}$ term, we ensure that $V_{\beta}(\Psi)\subseteq[0,1]$ whenever $\Psi(n)\to 0.$ The case where $\Psi(n)$ does not tend to zero is trivial, so we will always assume that $V_{\beta}(\Psi)\subseteq[0,1].$ It is important to note that for any $x\in[0,1]$ there exists a sequence $(\epsilon_{i})_{i=1}^{\infty}\in\{0,1\}^{\mathbb{N}}$ such that
\begin{equation}
\label{coding equation}
x=(\beta-1)\sum_{i=1}^{\infty}\frac{\epsilon_{i}}{\beta^{i}}.
\end{equation} This is a simple consequence of the fact that every element of $I_{\beta}$ has a $\beta$-expansion.

In this section we will prove the following theorem.
\begin{thm}
\label{Proved theorem}
Let $(\omega_{n})_{n=1}^{\infty}$ be a sequence of real numbers tending to infinity. Then for almost every $\beta\in(1.497\ldots,2)$ the set $V_{\beta}(\omega_{n}\cdot 2^{-n})$ is of full Lebesgue measure within $[0,1]$.
\end{thm} Theorem \ref{Proved theorem} is equivalent to Theorem \ref{Main theorem}. So to conclude our main result we just have to prove Theorem \ref{Proved theorem}.

We now take the opportunity to recall some of the results of Benjamini and Solomyak from \cite{Solomyak1}. The counting estimates provided in this paper will be essential in our proof of Theorem \ref{Proved theorem}. Benjamini and Solomyak study how the set $$A_{n}(\beta):=\Big\{(\beta-1)\sum_{i=1}^{n}\frac{\epsilon_{i}}{\beta^{i}}:(\epsilon_{i})\in\{0,1\}^{n}\Big\}$$ is distributed within $[0,1].$ $A_{n}(\beta)$ is precisely the set of level $n$ sums normalised by a factor $(\beta-1).$ Given $\beta\in(1,2)$, $s>0$ and $n\in\mathbb{N},$ Benjamini and Solomyak introduced the set
$$P(\beta,s,n):=\Big\{(a,b)\in A_{n}(\beta)^{2}: a\neq b \textrm{ and }|a-b|\leq \frac{s}{2^{n}}\Big\}.$$
They conjectured that for almost every $\beta\in(1,2)$ there exists $c,C>0$ such that $$cs\leq \frac{\# P(\beta,s,n)}{2^{n}}\leq Cs$$ for all $n\in\mathbb{N}$ and $s>0$. Within \cite{Solomyak1} some results are proved in the direction of this conjecture. For our purposes we only need one result which is the following.

\begin{thm}[Theorem 2.1 from \cite{Solomyak1}]
\label{Solomyak}
There exists $C_{1}>0$ such that $$\int_{(1.497\ldots,2)}\frac{\# P(\beta,s,n)}{2^{n}}\leq C_{1}s$$ for all $n\in\mathbb{N}$ and $s>0$.
\end{thm}Importantly the $C_{1}$ appearing in Theorem \ref{Solomyak} does not depend on $n$ or $s$. Theorem \ref{Solomyak} is a slightly weaker version of Theorem 2.1 from \cite{Solomyak1}, but we only require this weaker statement.

\subsection{Proof of Theorem \ref{Proved theorem}}
For the rest of this section we fix a sequence $(\omega_{n})_{n=1}^{\infty}$ that tends to infinity. For ease of exposition we let $\mathcal{I}:=(1.497\ldots,2)$. Our proof of Theorem \ref{Proved theorem} will be via a proof by contradiction. So let us assume that Theorem \ref{Proved theorem} is false and examine the implications of this assumption.

Given $\beta\in \mathcal{I}$ let
\begin{align*}
Bad(\beta):=\Big\{x\in [0,1]: &\textrm{ There exists finitely many solutions to the inequalities }\\
& 0\leq x-(\beta-1)\sum_{i=1}^{n}\frac{\epsilon_{i}}{\beta^{i}} \leq \frac{\omega_{n}}{2^{n}}\Big\}.
\end{align*}Similarly, given $l\in\mathbb{N}$ we let
\begin{align*}
Bad(\beta,l):=\Big\{x\in [0,1]: &\textrm{ For all }n\geq l \textrm{ there are no solutions to the inequalities}\\
& 0\leq x-(\beta-1)\sum_{i=1}^{n}\frac{\epsilon_{i}}{\beta^{i}} \leq \frac{\omega_{n}}{2^{n}}\Big\}.
\end{align*} Clearly $Bad(\beta)=\cup_{l=1}^{\infty} Bad(\beta,l)$ and $Bad(\beta,l)\subseteq Bad(\beta,l+1)$.
We also introduce the set $$\mathcal{E}:=\Big\{\beta\in\mathcal{I}: \lambda(Bad(\beta))>0\Big\},$$ and to each $l\in\mathbb{N}$ we introduce the following subset of $\mathcal{E}:$
$$\mathcal{E}_{l}:=\Big\{\beta\in\mathcal{I}: \lambda(Bad(\beta,l))>0\Big\}.$$ Remember we are assuming that Theorem \ref{Proved theorem} is false, this is precisely the statement that $\lambda(\mathcal{E})>0$. The following properties are important
\begin{equation}
\label{Decomposition of E}
\mathcal{E}=\bigcup_{l=1}^{\infty}\mathcal{E}_{l}\textrm{ and } \mathcal{E}_{l}\subseteq \mathcal{E}_{l+1}.
\end{equation}The fact that $\cup_{l=1}^{\infty}\mathcal{E}_{l}\subseteq \mathcal{E}$ is obvious. To see why $\mathcal{E}\subseteq \cup_{l=1}^{\infty}\mathcal{E}_{l}$ assume $\beta^{*}\in \mathcal{E},$ so $\lambda(Bad(\beta^{*}))>0$. Since $Bad(\beta^{*})=\cup_{l=1}^{\infty} Bad(\beta^{*},l)$ and $Bad(\beta^{*},l)\subseteq Bad(\beta^{*},l+1),$ there must exists $l^{*}\in\mathbb{N}$ sufficiently large for which $\lambda(Bad(\beta^{*},l^{*}))>0,$ therefore $\beta^{*}\in \mathcal{E}_{l^{*}}.$ It follows from (\ref{Decomposition of E}) and our assumption that $\lambda(\mathcal{E})>0,$ that there exists $L\in\mathbb{N}$ for which $\lambda(\mathcal{E}_{L})>0.$ This value of $L$ will be fixed for the rest of our proof.

 The proof of Theorem \ref{Proved theorem} will rely on an application of the Lebesgue density theorem. This theorem states that if $E\subseteq \mathbb{R}$ is a measurable set, then for almost every $x\in E$ the following holds
\begin{equation}
\label{density equation}
\lim_{r\to 0}\frac{\lambda(E\cap [x-r,x+r])}{2r}=1.
\end{equation} We call any $x\in E$ satisfying (\ref{density equation}) a \emph{density point for }$E$. To each $\beta\in \mathcal{I}$ we associate the set $$Bad_{d}(\beta,L):=\Big\{x\in Bad(\beta,L): x \textrm{ is a density point for } Bad(\beta,L)\Big\}.$$ Clearly $Bad_{d}(\beta,L)\subseteq Bad(\beta,L)$, the Lebesgue density theorem implies $\lambda(Bad_{d}(\beta,L))=\lambda( Bad(\beta,L)).$ Importantly if $\beta\in \mathcal{E}_{L}$ both of these quantities are positive.
Let $$A:=\int_{\mathcal{I}}\lambda(Bad_{d}(\beta,L))d\lambda(\beta).$$ It is a consequence of $\mathcal{E}_{L}$ having positive measure that $A>0.$ The following lemma constructs a $y\in[0,1]$ for which $y\in Bad_{d}(\beta,L)$ for a large subset of $\mathcal{I}.$ Before the statement and proof of this lemma we introduce a piece of notation, given $x\in[0,1]$ let $\widehat{Bad}(x,L):=\{\beta\in \mathcal{I}: x\in Bad_{d}(\beta,L)\}.$

\begin{lemma}
\label{Construction of useful y}
There exists $y\in[0,1]$ such that $\lambda(\widehat{Bad}(y,L))\geq A$.
\end{lemma}
\begin{proof}
This lemma is a straightforward consequence of Fubini's theorem but we include the details for completion. We observe the following:
\begin{align*}
A&=\int_{\mathcal{I}}\lambda(Bad_{d}(\beta,L))d\lambda(\beta)\\
&=\int_{\mathcal{I}}\int_{[0,1]}\mathbbm{1}_{Bad_{d}(\beta,L)}(x)d\lambda(x)d\lambda(\beta)\\
&=\int_{[0,1]}\int_{\mathcal{I}}\mathbbm{1}_{\widehat{Bad}(x,L)}(\beta)d\lambda(\beta)d\lambda(x)\textrm{  (By Fubini's theorem) }\\
&=\int_{[0,1]}\lambda(\widehat{Bad}(x,L))d\lambda(x).
\end{align*}

At least one $x\in[0,1]$ must satisfy $\lambda(\widehat{Bad}(x,L))\geq A.$ Otherwise $\int_{[0,1]}\lambda(\widehat{Bad}(x,L))d\lambda(x)<A,$ which is not possible.
\end{proof}

For the rest of our proof the value $y\in[0,1]$ will always denote the element we have constructed in Lemma \ref{Construction of useful y}. We now define a collection of subsets of $\widehat{Bad}(y,L),$ given $m\in\mathbb{N}$ let
\begin{align*}
\widehat{Bad}(y,L,m):=\Big\{\beta\in \widehat{Bad}(y,L):& \textrm{ For all }n\geq m \textrm{ we have } \\ &\frac{\lambda(Bad(\beta,L)\cap[y-\frac{1}{\beta^{n}},y+\frac{1}{\beta^{n}}])}{2\beta^{-n}}> 1-\frac{\lambda(A)}{20C_{1}}\Big\}.
\end{align*}
For each $\beta\in \widehat{Bad}(y,L)$ the element $y$ is a density point for $Bad(\beta,L)$, therefore $\widehat{Bad}(y,L)= \cup_{m=1}^{\infty}\widehat{Bad}(y,L,m).$ Since $\widehat{Bad}(y,L,m)\subseteq \widehat{Bad}(y,L,m+1),$ we can take $M\in\mathbb{N}$ sufficiently large that
\begin{equation}
\label{Uniform equation}
\lambda(\widehat{Bad}(y,L,M))\geq \lambda(A)/2.
\end{equation}This quantity $M$ will be fixed for the rest of our proof. The appearance of the quantity $\lambda(A)(20C_{1})^{-1}$ within the definition of $\widehat{Bad}(y,L,m)$ is due to technical reasons that will become clear when we finish our proof of Theorem \ref{Proved theorem}. Moreover the parameter $C_{1}$ is the same number appearing in the statement of Theorem \ref{Solomyak}.

We now define a new subset of $\mathcal{A}_{n}(\beta)$. Given $\beta\in(1,2)$, $s>0,$ and $n\in\mathbb{N}$, let $$T(\beta,s,n):=\Big\{a\in A_{n}(\beta): \exists b\in A_{n}(\beta) \textrm{ satisfying } a\neq b\textrm{ and } |a-b|\leq \frac{s}{2^{n}}\Big\}.$$ If $T(\beta,s,n)$ is a small set, then most elements of $A_{n}(\beta)$ will be separated by a factor $s\cdot 2^{-n}$. We will show that for certain values of $\beta$, $s$, and $n$ we have such a property, we will then use this property in our later calculations. To obtain bounds on the size of the set $T(\beta,s,n)$, we relate the size of $T(\beta,s,n)$ with the size of $P(\beta,s,n).$ We will then be able to use the machinery of Benjamini and Solomyak.
\begin{lemma}
\label{Counting bounds}
$\# T(\beta,s,n)\leq \# P(\beta,s,n).$
\end{lemma}
\begin{proof}
The proof of this statement is straightforward but we include it for completeness. Suppose $a\in T(\beta,s,n)$, then there exists $b\in A_{n}(\beta)$ such that $(a,b)\in P(\beta,s,n)$ and $(b,a)\in P(\beta,s,n).$ So if $a\in T(\beta,s,n)$ it appears as either the first or second coordinate of at least two elements of $P(\beta,s,n)$. The collection of numbers that appear as either the first or second coordinate of an element of $P(\beta,s,n)$ has cardinality $2\# P(\beta,s,n).$ Here we have allowed for a number to be counted multiple times. By the above each $a\in T(\beta,s,n)$ appears at least twice within this collection, therefore $2\# T(\beta,s,n)\leq 2\# P(\beta,s,n).$ This implies $\# T(\beta,s,n)\leq \# P(\beta,s,n).$
\end{proof}
In the following proposition $s$ is a fixed value. For this value of $s$ we will construct for each $n\in\mathbb{N}$ a useful $\beta\in\widehat{Bad}(y,L,M)$ for which $T(\beta,s,n)$ is small.

\begin{prop}
\label{Separation prop}
For each $n\in\mathbb{N}$ there exists $\beta_{n}\in \widehat{Bad}(y,L,M)$ such that $\# T(\beta_{n},\frac{\lambda(A)}{5C_{1}},n)< 2^{n-1}.$
\end{prop}
\begin{proof}
In this proof we will obtain bounds on the Lebesgue measure of the set
$$H_{n}:=\Big\{\beta\in \mathcal{I}: \# T\Big(\beta, \frac{\lambda(A)}{5C_{1}},n\Big)
\geq 2^{n-1}\Big\}.$$ We will show that for any $n\in\mathbb{N}$ we have $\lambda(H_{n})<\lambda(\widehat{Bad}(y,L,M)),$ thus implying our result. In the following we make use of the estimates provided by Theorem \ref{Solomyak};
\begin{align*}
\lambda(H_{n})& \leq \lambda\Big(\beta\in\mathcal{I}: \# P\Big(\beta, \frac{\lambda(A)}{5C_{1}},n\Big)\geq 2^{n-1}\Big)\textrm{  (By Lemma } \ref{Counting bounds})\\
&= \frac{2^{n}\lambda\Big(\beta\in\mathcal{I}: \# P\Big(\beta, \frac{\lambda(A)}{5C_{1}},n\Big)\geq 2^{n-1}\Big)}{2^{n}}\\
&\leq 2\int_{\beta\in\mathcal{I}: \# P\Big(\beta, \frac{\lambda(A)}{5C_{1}},n\Big)\geq 2^{n-1}}\frac{\# P\Big(\beta, \frac{\lambda(A)}{5C_{1}},n\Big)}{2^{n}}d\lambda(\beta)\\
&\leq 2\int_{\mathcal{I}}\frac{\# P\Big(\beta, \frac{\lambda(A)}{5C_{1}},n\Big)}{2^{n}}d\lambda(\beta)\\
&\leq \frac{2C_{1}\lambda(A)}{5C_{1}} (\textrm{By Theorem }\ref{Solomyak})\\
&< \frac{\lambda(A)}{2}.
\end{align*}
By (\ref{Uniform equation}) we know that $\lambda(\widehat{Bad}(y,L,M))\geq \lambda(A)/2,$ so we can conclude our result.
\end{proof}
We emphasise that the $\beta_{n}$ we construct in Proposition \ref{Separation prop} depends on $n$. Also note that without loss of generality we may assume that $\beta_{n}$ is transcendental. We now have all the necessary tools to prove Theorem \ref{Proved theorem}

\begin{proof}[Proof of Theorem \ref{Proved theorem}]

We begin by taking $N\in\mathbb{N}$ to be some natural number sufficiently large that
\begin{equation}
\label{N large}
L\leq N+M,
\end{equation}and
\begin{equation}
\label{growth equation}
\frac{\lambda(A)2^{M}}{5C_{1}\beta^{M}}\leq \omega_{N+M}
\end{equation}for all $\beta\in\mathcal{I}.$ Recall that $M$ is the parameter appearing in (\ref{Uniform equation}). We can pick an $N$ that satisfies (\ref{growth equation}) because $(\omega_{n})_{n=1}^{\infty}$ tends to infinity. Applying Proposition \ref{Separation prop} with $N$ as above, there exists $\beta_{N}\in \widehat{Bad}(y,L,M),$ which we may assume to be transcendental, for which the set
$$S:=\{a\in A_{N}(\beta_{N}): \textrm{ such that }  |a-b|> \frac{\lambda(A)}{5C_{1}2^{N}} \textrm{ for all }b\in A_{n}(\beta_{N})\setminus\{a\}\Big\}$$ satisfies \begin{equation}
\label{S size}
\# S\geq 2^{N-1}.
 \end{equation}For each $a\in S$ there exists $(\epsilon_{i})_{i=1}^{N}\in\{0,1\}^{N}$ such that $(\beta_{N}-1)\sum \epsilon_{i}\beta_{N}^{-i}=a.$ Let $\{(\epsilon_{i}^{j})_{i=1}^{N}\}$ denote the set of these sequences. By (\ref{S size}) this set must contain at least $2^{N-1}$ elements.

 To each $(\epsilon_{i}^{j})_{i=1}^{N}$ we associate the interval
$$\mathcal{I}^{j}:=\Big[(\beta_{N}-1)\sum_{i=1}^{N} \frac{\epsilon_{i}^{j}}{\beta_{N}^{i}},(\beta_{N}-1)\sum_{i=1}^{N} \frac{\epsilon_{i}^{j}}{\beta_{N}^{i}} + \frac{\lambda(A)}{5C_{1}2^{N}}\Big].$$ Without loss of generality $\mathcal{I}^{j}\subseteq [0,1]$ for each $(\epsilon_{i}^{j})_{i=1}^{N}.$ Moreover, since each element of $S$ is separated by a factor of at least $\lambda(A)(5C_{1}2^{N})^{-1},$ we have $\mathcal{I}^{j}\cap \mathcal{I}^{j'}=\emptyset$ for each $j\neq j'.$

Each interval $\mathcal{I}^{j}$ is contained within the interval $[0,1],$ we have no information about how they are positioned except that they are all mutually disjoint. To achieve our desired contradiction we need to take these intervals and make them somehow local to $y$. This we do below.

Let $(\delta_{i})_{i=1}^{\infty}\in\{0,1\}^{\mathbb{N}}$ satisfy $y=(\beta_{N}-1)\sum_{i=1}^{\infty}\delta_{i}\beta_{N}^{-i}$. Such a sequence exists by (\ref{coding equation}). For any $m\in\mathbb{N}$ we have
\begin{equation}
\label{interval equation}
(\beta_{N}-1)\sum_{i=1}^{m}\frac{\delta_{i}}{\beta_{N}^{i}}\in\Big[y-\frac{1}{\beta_{N}^{m}},y\Big].
\end{equation} Given an interval $\mathcal{I}^{j}$ we now define the following local scaled interval
\begin{align*}
\tilde{\mathcal{I}}^{j}&:= (\beta_{N}-1)\sum_{i=1}^{M}\frac{\delta_{i}}{\beta_{N}^{i}} + \frac{\mathcal{I}^{j}}{\beta_{N}^{M}}\\
&=\Big[(\beta_{N}-1)\Big(\sum_{i=1}^{M}\frac{\delta_{i}}{\beta_{N}^{i}}+\sum_{i=1}^{N}\frac{\epsilon_{i}^{j}}{\beta_{N}^{M+i}}\Big),(\beta_{N}-1)\Big(\sum_{i=1}^{M}\frac{\delta_{i}}{\beta_{N}^{i}}+\sum_{i=1}^{N}\frac{\epsilon_{i}^{j}}{\beta_{N}^{M+i}}\Big)+\frac{\lambda(A)}{5C_{1}\beta_{N}^{M}2^{N}}\Big]
\end{align*} These intervals satisfy the following properties:
\begin{align}
\label{prop1}
\tilde{\mathcal{I}}^{j}\cap \tilde{\mathcal{I}}^{j'}=\emptyset \textrm{ for all } j\neq j',\\
\label{prop2}
\lambda(\tilde{\mathcal{I}}^{j})=\frac{\lambda(A)}{5C_{1}\beta_{N}^{M}2^{N}},\\
\label{prop3}
\tilde{\mathcal{I}}^{j}\subset [y-\beta_{N}^{-M},y+\beta_{N}^{-M}].
\end{align}
Equation (\ref{prop1}) follows from $\mathcal{I}^{j}\cap \mathcal{I}^{j'}=\emptyset$, equation (\ref{prop2}) is obvious, and equation (\ref{prop3}) is a consequence of equation (\ref{interval equation}). Moreover (\ref{growth equation}) implies that
$$\frac{\lambda(A)}{5C_{1}\beta_{N}^{M}2^{N}}\leq \frac{\omega_{N+M}}{2^{N+M}}.$$ Therefore, since $N+M\geq L$ we have
\begin{equation}
\label{complement inclusion}
\tilde{\mathcal{I}}^{j}\subseteq Bad^{c}(\beta_{N},L).
\end{equation}
Here $Bad^{c}(\beta_{N},L)$ is simply the complement of $Bad(\beta_{N},L)$. Combining the above properties we have the following estimates on the normalised Lebesgue measure of $Bad(\beta_{N},L)$ within the interval $[y-\beta_{N}^{-M},y+\beta_{N}^{-M}]:$

\begin{align*}
\frac{\lambda(Bad(\beta_{N},L)\cap [y-\beta_{N}^{-M},y+\beta_{N}^{-M}])}{2\beta_{N}^{-M}}&=1 - \frac{\lambda(Bad^{c}(\beta_{N},L)\cap [y-\beta_{N}^{-M},y+\beta_{N}^{-M}])}{2\beta_{N}^{-M}}\\
&\leq 1- \frac{\lambda(\bigcup \tilde{\mathcal{I}}^{j})}{2\beta_{N}^{-M}}\, (\textrm{By }(\ref{prop3})\textrm { and }(\ref{complement inclusion})) \\
& = 1- \frac{\sum\lambda(\tilde{\mathcal{I}}^{j})}{2\beta_{N}^{-M}}\, (\textrm{By }(\ref{prop1}))\\
& \leq  1- \frac{2^{N-1}\lambda(A)}{10C_{1}\beta_{N}^{M}2^{N}\beta_{N}^{-M}}(\textrm{By }(\ref{S size})\textrm { and }(\ref{prop2}))\\
&= 1- \frac{\lambda(A)}{20C_{1}}.
\end{align*}

But this contradicts the condition $\beta_{N}\in \widehat{Bad}(y,L,M).$ So we have our contradiction and have proved Theorem \ref{Proved theorem}.
\end{proof}

\begin{remark}
It is natural to ask why the proof of Theorem \ref{Proved theorem} cannot be extended to show that for almost every $\beta\in \mathcal{I}$ the set $V_{\beta}(2^{-n})$ is of full measure. We cannot prove this stronger statement as a consequence of the way we scale the intervals $\mathcal{I}^{j}$ to the local intervals $\tilde{\mathcal{I}}^{j}.$ The natural way to scale the intervals $\mathcal{I}^{j}$ is to use the expansion of $y,$ as we do in our proof our Theorem \ref{Proved theorem}. However, this method takes intervals of the order $c\cdot 2^{-N}$ and gives intervals of the order $c\cdot \beta^{-M}2^{-N}$ which are contained within an interval of size $d\cdot \beta^{-M}.$ To prove $V_{\beta}(2^{-n})$ is of full measure, we would need intervals of the order $c\cdot 2^{-(N+M)}$ which are within an interval of size $d\cdot 2^{-M}.$ In our proof of Theorem \ref{Proved theorem} we used the fact that $(w_{n})_{n=1}^{\infty}$ tends to infinity, to overcome the inefficiencies in the way we scale the intervals $\mathcal{I}^{j}$ to the local intervals $\tilde{\mathcal{I}}^{j}$.

\end{remark}
\begin{remark}
By Theorem \ref{Main theorem} we know that for a typical $\beta\in\mathcal{I}$ the set $W_{\beta}(\omega_{n}\cdot 2^{-n})$ is of full measure, but what can we say about the exceptional set $I_{\beta}\setminus W_{\beta}(\omega_{n}\cdot 2^{-n})$. As we will see, the following set is always contained in $I_{\beta}\setminus W_{\beta}(\omega_{n}\cdot 2^{-n}),$ whenever $(\omega_{n})_{n=1}^{\infty}$ grows sufficiently slowly. 

Let $$U_{\beta}:=\Big\{x\in \Big(0,\frac{1}{\beta-1}\Big):x \textrm{ has a unique }\beta\textrm{-expansion}\Big\}.$$ Understanding the properties of the set $U_{\beta}$ is a classical problem within expansions in non-integer bases. By the work of Dar\'{o}czy and K\'{a}tai \cite{DaKa}, and Erd\H{o}s, Jo\'{o} and Komornik \cite{Erdos2}, it is known that $U_{\beta}$ is nonempty if and only if $\beta\in(\frac{1+\sqrt{5}}{2},2).$ Moreover, in \cite{GlenSid} Glendinning and Sidorov showed that:
\begin{itemize}
  \item $U_{\beta}$ is countable if $\beta\in(\frac{1+\sqrt{5}}{2},\beta_{c})$,
  \item $U_{\beta_{c}}$ is uncountable with zero Hausdorff dimension,
  \item $U_{\beta}$ has strictly positive Hausdorff dimension if $\beta\in(\beta_{c},2)$.
\end{itemize}
Here $\beta_{c}\approx 1.78723$ is the the Komornik-Loreti constant introduced in \cite{KomLor}.

The connection between the set $U_{\beta}$ and the set $I_{\beta}\setminus W_{\beta}(\omega_{n}\cdot 2^{-n})$ is seen through the following inequalities. For each $x\in U_{\beta}$ there exists $\kappa(x)>0$ such that
\begin{equation}
\label{kappa equation}
\frac{\kappa(x)}{\beta^{n}(\beta-1)}\leq x-\sum_{i=1}^{n}\frac{\epsilon_{i}}{\beta^{i}}\leq \frac{1}{\beta^{n}(\beta-1)}
\end{equation}
for all $n\in\mathbb{N}$. Here $(\epsilon_{i})_{i=1}^{\infty}$ is the unique $\beta$-expansion of $x$. As a consequence of (\ref{kappa equation}), if $w_{n}\cdot 2^{-n}=o(\beta^{-n})$ then $U_{\beta}\subseteq I_{\beta}\setminus W_{\beta}(\omega_{n}\cdot 2^{-n}).$ Therefore, Theorem \ref{Main theorem} combined with the aforementioned results of Glendinning and Sidorov, imply that there exists $\beta\in(\beta_{c},2)$ for which $W_{\beta}(\omega_{n}\cdot 2^{-n})$ has full Lebesgue measure, yet $I_{\beta}\setminus W_{\beta}(\omega_{n}\cdot 2^{-n})$ is a set of positive Hausdorff dimension.
\end{remark}

\section{The convergence case and the mass transference principle principle}
In this section we discuss the case where $\sum_{n=1}^{\infty}2^{n}\Psi(n)<\infty$ and point out some consequences of the work presented here and in \cite{Bak}. To obtain these results we make use of the mass transference principle principle of Beresnevich and Velani \cite{BerVel}. Before we state the result from \cite{BerVel} that we use, we need to introduce some notation. Given a ball $B=B(x,r)$ contained in $\mathbb{R},$ let $B^{s}:=B(x,r^{s})$ where $s$ is any parameter within $(0,1)$. The following theorem is a version of Theorem 3 from \cite{BerVel} that we have rewrote to suit our purposes.

\begin{thm}
\label{Mass transference principle}
Let $(B_{i})_{i=1}^{\infty}$ be a sequence of balls in $I_{\beta}$ such that $r(B_{i})\to 0$. Suppose that for any ball $B\subseteq I_{\beta}$ we have $$\lambda\Big(B\bigcap \limsup B_{i}^{s}\Big)=\lambda(B).$$ Then for any ball $B\subseteq I_{\beta}$ we have $$\mathcal{H}^{s}\Big(B\bigcap \limsup B_{i}\Big)=\mathcal{H}^{s}(B).$$
\end{thm}Here $\mathcal{H}^{s}$ is the $s$-dimensional Hausdorff measure. A version of Theorem \ref{Mass transference principle} holds for more general dimension functions, in this article we focus only on Hausdorff measure and Hausdorff dimension.

Combining Theorem \ref{Mass transference principle} with Theorem \ref{Garsia thm} we have the following result.

\begin{thm}
\label{Hausdorff Garsia}
Let $\beta\in(1,2)$ be a Garsia number and $\Psi:\mathbb{N}\to\mathbb{R}_{\geq 0}$ satisfy $\Psi(n)\to 0$. Then the following statements are true:
$$\sum_{n=1}^{\infty}2^{n}\Psi(n)^{s}<\infty \Rightarrow \mathcal{H}^{s}(W_{\beta}(\Psi))=0$$ and
$$\sum_{n=1}^{\infty}2^{n}\Psi(n)^{s}=\infty \Rightarrow \mathcal{H}^{s}(W_{\beta}(\Psi))=\infty.$$
\end{thm}The first part of Theorem \ref{Hausdorff Garsia} follows from a covering argument. The second statement is where we use the mass transference principle. As an application of Theorem \ref{Hausdorff Garsia} we obtain the following result.

\begin{thm}
\label{root 2 theorem}
Let $\alpha\in(1,\infty)$ and $\Psi(n)=2^{-n\alpha},$ then $\mathcal{H}^{1/\alpha}(W_{\sqrt{2}}(\Psi))=\infty.$
\end{thm}This follows since $\sqrt{2}$ is a Garsia number. Note that Theorem \ref{root 2 theorem} provides an explicit example of a $\beta$ which satisfies the almost everywhere statement appearing in Theorem \ref{PerRev thm}. Theorem \ref{root 2 theorem} also yields the Hausdorff measure at the dimension, something which is not covered by Persson and Reeve. Importantly our results do not imply the intersection properties included in Theorem \ref{PerRev thm}.

Similarly, if we combine Theorem \ref{Mass transference principle} with Theorem \ref{Main theorem} we obtain the following weaker version of Theorem \ref{PerRev thm}.

\begin{thm}
\label{Hausdorff ae}
For almost every $\beta\in(1.49\ldots,2),$ we have that $\dim_{H}(W_{\beta}(2^{-n\alpha}))=1/\alpha$ for any $\alpha\in(1,\infty).$
\end{thm}We omit the proof of Theorem \ref{Hausdorff ae}. The proof follows from making a sensible choice of $(\omega_{n})_{n=1}^{\infty}$ which grows sufficiently slowly, $(\log n)_{n=1}^{\infty}$ works. We then combine Theorem \ref{Main theorem} and Theorem \ref{Mass transference principle}. Unfortunately this method does not provide any information on the Hausdorff measure at the dimension. A proof of Conjecture \ref{conjecture 1} would imply that for almost every $\beta\in(1,2)$ we have an analogue of Theorem \ref{Hausdorff Garsia}.
\smallskip

\noindent \textbf{Acknowledgements} Part of this paper was completed during the author's visit to LIAFA. The author is grateful to Wolfgang Steiner for providing an enjoyable working environment during this visit. The author would also like to thank Karma Dajani and Tom Kempton for being a good source of discussion. Additional thanks should go to Tom Kempton for introducing us to \cite{Solomyak1} and for posing several interesting questions.

\end{document}